\documentclass[11pt]{article}
\usepackage[margin=1.15in]{geometry}
\usepackage{amsmath,amssymb,amsthm,mathrsfs}
\usepackage{hyperref}
\hypersetup{hidelinks}
\usepackage{booktabs}
\usepackage{listings}
\theoremstyle{plain}
\newtheorem{theorem}{Theorem}[section]
\newtheorem{proposition}[theorem]{Proposition}
\newtheorem{lemma}[theorem]{Lemma}
\newtheorem{corollary}[theorem]{Corollary}
\theoremstyle{definition}
\newtheorem{definition}[theorem]{Definition}
\newtheorem{remark}[theorem]{Remark}
\newcommand{\Z}{\mathbf Z}
\newcommand{\Zp}{\Z_p}
\newcommand{\Rq}{R_q}
\newcommand{\PGL}{\operatorname{PGL}}
\newcommand{\Mgp}{\mathfrak M}

\title{Principal-unit Di\c{t}\u{a}-type Butson matrices:\\
character-table transitions, fractional-linear automorphisms,\\
Weyl thresholds, and polarized Heisenberg windows}
\author{Richard Scarlini}
\date{}

\begin{document}
\maketitle

\begin{abstract}
For an odd prime \(p\) and integers \(r,q\ge1\) we study the Butson-type complex
Hadamard matrices \(H_{p;r,q}(a,b)=\chi(1+p^{r}ab)\) on \(\Z/p^{q}\Z\), where
\(\chi\) is a faithful character of the principal units
\((1+p^{r}\Zp)/(1+p^{q+r}\Zp)\). The defining phase is the normalized \(p\)-adic
logarithm, which is linear modulo \(p^{q}\) exactly when \(q\le r\) and genuinely
nonlinear once \(q>r\). We prove that this single transition governs four
independent structural properties. The matrix is monomially equivalent to the
character table of a finite abelian group if and only if \(q\le r\); a canonical
cyclic \(2\)-cochain is a cocycle if and only if \(q\le r\); the canonical
M\"obius torsors have projectively scalar Weyl commutator if and only if
\(q\le r\); and in each case the obstruction is the same expression \(p^{r}\)
times a polynomial taking unit values.  Beyond that global Weyl threshold,
when \(q>r\), scalar commutators survive on exactly \(q-r+1\) inclusion-maximal
valuation-polarized windows.  After quotienting their radicals, every such
window carries the same perfect pairing on two cyclic groups of order \(p^r\),
hence the standard finite Heisenberg extension of order \(p^{3r}\).  We classify
the exact normalizer of every maximal window inside the projective monomial
group, prove that distinct polarizations are pairwise nonconjugate there, and
identify the residual unit-torus action through
\((\Z/p^r\Z)^{\times}\).  Our main theorem is a complete classification of the
projective monomial automorphism group: it is the three-parameter
fractional-linear group of order \(p^{2q}\varphi(p^{q})\) described in
Theorem~\ref{thm:auto}, whose compatible inverse limit over \(q\) is the
\(\Gamma_0(p^{r})\)-type congruence subgroup of \(\PGL_2(\Zp)\).
The depth-\(5\) specialization is the quotient-profile family arising in the
Hensel-shell spectral problem; the present paper is the publication owner of the
complete principal-unit projective-monomial classification. As an
application we separate \(H_{p;r,q}\) from the Fourier matrix \(F_{p^{q}}\) for
\(q>r\) by abelianization, an invariant strictly finer than the order of the
automorphism group, which for these two matrices coincides.

The family is recursively Di\c{t}\u{a}-type, and we make no claim that its
members lie outside known equivalence orbits; the order-\(9\) member is a
distinguished arithmetic point of the known \(F_9^{(4)}\) deformation family.
Novelty is claimed for the explicit formula and for the classification results.
For the next nonlinear ternary member, a reproducible computation gives
\(\operatorname{def}(H_{3;1,3})=46\), compared with
\(\operatorname{def}(F_{27})=28\); their complete \(2\times2\)
fingerprint slices nevertheless agree. Among the four nonlinear members we
compute, the defect separates two and the fingerprint separates none, while the
abelianization separates all of them; this is our reason for using it.
\end{abstract}

\section{Introduction}\label{sec:intro}

\subsection{The family}

Let \(p\) be an odd prime and \(r,q\ge1\) integers. Write \(\Rq=\Z/p^{q}\Z\) and
\(\varepsilon=p^{r}\). The principal units \(1+p^{r}\Zp\) form a procyclic group
and
\[
\bigl(1+p^{r}\Zp\bigr)\big/\bigl(1+p^{q+r}\Zp\bigr)\;\cong\;\Z/p^{q}\Z .
\]
Fix a faithful character \(\chi\) of that quotient, valued in the \(p^{q}\)-th
roots of unity.

\begin{definition}\label{def:family}
The \emph{principal-unit kernel} is the \(p^{q}\times p^{q}\) matrix
\[
H_{p;r,q}(a,b)=\chi\bigl(1+\varepsilon ab\bigr),\qquad a,b\in \Rq .
\]
Equivalently, writing
\(\ell_{r,q}(z)=p^{-r}\log\bigl(1+p^{r}z\bigr)\bmod p^{q}\)
for the normalized \(p\)-adic logarithm and \(\omega=e^{2\pi i/p^{q}}\), we have
\(H_{p;r,q}(a,b)=\omega^{\,\ell_{r,q}(ab)}\).
\end{definition}

Two elementary observations fix the shape of everything that follows. First,
\(\ell_{r,q}\) is well defined: the series
\(\ell_{r,q}(z)=z-\tfrac{p^{r}}{2}z^{2}+\tfrac{p^{2r}}{3}z^{3}-\cdots\)
converges \(p\)-adically and depends only on \(z\bmod p^{q}\). Second, and this
is the whole subject of the paper,
\[
\ell_{r,q}(z)\equiv z \pmod{p^{q}}
\quad\Longleftrightarrow\quad q\le r ,
\]
since the first correction term \(-p^{r}z^{2}/2\) survives modulo \(p^{q}\)
precisely when \(q>r\) (recall \(p\) is odd, so \(2\) is invertible).

\subsection{Origin in the Hensel-shell profile and theorem ownership}\label{ssec:origin}
The family is studied here independently for arbitrary depth \(r\), but one
natural occurrence motivated the construction.  In the Hensel-shell spectral
problem of the companion Paper F-I~\cite{PaperFICompanion}, double dephasing of a
fixed active support block produces exactly
\[
H_a(b)=\chi_q(1+p^5ab),
\]
which is the specialization \(H_{p;5,q}\) after the quotient coordinates are
identified with \(R_q\).  Paper F-I proves that model-specific realization and
the regular torsor/scalarity consequences needed for its shell operators.

The complete projective-monomial classification is proved only here, in the
general \(r\)-parameter family.  Paper F-I records a short \(r=5\) corollary but
does not duplicate the proof.  Thus the logical dependency is one-way at theorem
level: the Hensel calculation supplies an occurrence of the kernel, while the
classification below is self-contained and does not rely on the specialized
Paper F argument.

\subsection{Results}

We prove the following. Throughout, \emph{monomial equivalence} means the
two-sided action \(H\mapsto D_1PHQD_2\) by permutation and unimodular diagonal
matrices, in the sense of \cite{Ostergard}, and a \emph{projective monomial
automorphism} is such a pair fixing \(H\) up to a global scalar.

\begin{enumerate}
\item[(I)] \(H_{p;r,q}\) is a \(BH(p^{q},p^{q})\) Butson matrix
(Theorem~\ref{thm:butson}).
\item[(II)] It is monomially equivalent to the character table of a finite
abelian group if and only if \(q\le r\) (Theorem~\ref{thm:chartable}).
\item[(III)] Its projective monomial automorphism group is the fractional-linear
group \(\Mgp_{p;r,q}\) of order \(p^{2q}\varphi(p^{q})\), with
\(\Gamma_0(p^{r})\)-type profinite limit (Theorem~\ref{thm:auto}).
\item[(IV)] The canonical cyclic \(2\)-cochain is a cocycle if and only if
\(q\le r\), with no coboundary repair (Theorem~\ref{thm:cocycle}).
\item[(V)] The canonical M\"obius torsors have projectively scalar Weyl
commutator if and only if \(q\le r\) (Theorem~\ref{thm:weyl}).
\item[(VI)] For \(q>r\), the inclusion-maximal scalar-Weyl rectangles are exactly
\(\alpha+\beta=q-r\); each has a perfect residual pairing on
\((\Z/p^r\Z)^2\) and therefore a standard finite Heisenberg central extension
(Theorem~\ref{thm:polarized}).
\item[(VII)] The exact normalizer of each maximal polarization is the window
extended by diagonal units; distinct polarizations are pairwise nonconjugate
inside \(\Mgp_{p;r,q}\), and the effective residual action is
\((\Z/p^r\Z)^\times\) (Theorem~\ref{thm:normalizer}).
\end{enumerate}

Sections~\ref{sec:butson}--\ref{sec:consequences} prove (I)--(V),
Section~\ref{sec:polarized} proves (VI)--(VII), and
Section~\ref{sec:abelian} gives the abelianization separation from
\(F_{p^{q}}\).

\subsection{One threshold, seen three ways}\label{ssec:one}

The three sharp statements (II), (IV), (V) are not three coincidences. Each
reduces to the vanishing modulo \(p^{q}\) of \(p^{r}\) times a polynomial in the
parameters:
\[
\begin{array}{ll}
\text{character table:} & p^{r}(a_1-a_0)(a_2-a_0)\equiv0,\\[2pt]
\text{cocycle:} & p^{r}abc(a-c)\equiv0,\\[2pt]
\text{Weyl centrality:} & p^{r}cs\equiv0 .
\end{array}
\]
In each case the polynomial attains unit values, so the condition holds
identically exactly when \(p^{r}\equiv0\) in \(\Rq\), i.e.\ when \(q\le r\); and
in each case the first failure occurs at \(q=r+1\) with an explicit unit witness.
The global threshold is therefore a single transition seen through three
structures.  Theorems~\ref{thm:polarized} and~\ref{thm:normalizer} describe what
remains on the nonlinear side: global canonical Weyl centrality is lost, but a
finite valuation filtration of maximal scalar windows survives, each with the
same nondegenerate residual Heisenberg pairing and a rigid normalizer.

\subsection{Relation to known constructions}\label{ssec:known}

The family is \emph{recursively Di\c{t}\u{a}-type} \cite{Dita}. Writing
\(a=p^{q-r}A+i\) and \(b=j+p^{q-r}B\) one has the exact block identity
\[
\ell_{r,q}\bigl(a(j+p^{q-r}B)\bigr)\equiv\ell_{r,q}(aj)+p^{q-r}iB
\pmod{p^{q}},
\tag{1.1}\label{eq:dita}
\]
so each block is a diagonal phase times a Fourier matrix \(F_{p^{r}}\), and the
recursion iterates. The smallest nonlinear member, the order-\(9\) matrix at
\((p,r,q)=(3,1,2)\), lies in the known \(F_9^{(4)}\) Fourier/Di\c{t}\u{a}
deformation orbit and has defect \(10\)
\cite{Karlsson,TadejZyczkowskiDefect}; at that order it is thus a distinguished
arithmetic point in an already-mapped deformation geometry, and we assert no new
order-\(9\) equivalence class.

Our positioning is the following. Although the matrices admit an iterated
Di\c{t}\u{a} decomposition, we have not found the principal-unit kernel
\(\chi(1+p^{r}ab)\) identified as an all-parameter family in the existing
catalogues or construction literature
\cite{TadejZyczkowski,LampioOstergardSzollosi,DoDuc,DoDucSchmidt}; at small
orders its members may nevertheless belong to previously known equivalence
orbits. \emph{This is a negative literature-search conclusion, not a proof of
nonexistence}: the formula could be concealed by dephasing or reparametrization
at a particular order.

The two sides of the threshold have different conceptual homes. For \(q\le r\)
the kernel is biadditive and the family falls within the established theory of
Butson matrices from bilinear forms and group-invariant constructions over finite
local rings \cite{DoDuc,DoDucSchmidt}. For \(q>r\) the surviving logarithmic
terms make the phase non-biadditive on the additive index group, placing the
family outside that framework even though the matrices remain recursively
Di\c{t}\u{a}-type. Prior work thus explains the easy side; the content of
Theorem~\ref{thm:chartable} is an invariant proving impossibility on the other.

Complete automorphism groups of Hadamard families have been determined in
several cases --- Paley type II \cite{deLauneyStafford}, generalized Sylvester
\cite{EganFlannery} --- and projective-linear symmetry is not itself new in this
literature. What appears new in Theorem~\ref{thm:auto} is the finite-ring
M\"obius normal form together with its \(p\)-adic congruence limit.

\section{The Butson property}\label{sec:butson}

\begin{theorem}\label{thm:butson}
For every odd prime \(p\) and all \(r,q\ge1\), the matrix \(H=H_{p;r,q}\) has
entries that are \(p^{q}\)-th roots of unity and satisfies \(HH^{*}=p^{q}I\).
That is, \(H\) is a \(BH(p^{q},p^{q})\) matrix in the sense of Butson
\cite{Butson}.
\end{theorem}

\begin{proof}
The entries lie in \(\mu_{p^{q}}\) by construction. For \(b,b'\in\Rq\),
multiplicativity of \(\chi\) gives
\[
\sum_{a\in\Rq}H(a,b)\overline{H(a,b')}
=\sum_{a\in\Rq}\chi\!\left(\frac{1+\varepsilon ab}{1+\varepsilon ab'}\right)
=\sum_{a\in\Rq}\chi\bigl(1+\varepsilon m(a)\bigr),
\qquad
m(a)=\frac{a(b-b')}{1+\varepsilon ab'} .
\]
If \(b=b'\) every summand is \(1\) and the sum is \(p^{q}\). Suppose \(b\neq b'\)
and put \(\kappa=v_p(b-b')\), so \(0\le\kappa<q\). The denominator
\(1+\varepsilon ab'\) is a principal unit, so \(a\mapsto m(a)\) is \(p^{\kappa}\)
times a bijection of \(\Rq\) onto \(p^{\kappa}\Rq\), each value being attained
exactly \(p^{\kappa}\) times. Hence
\[
\sum_{a\in\Rq}\chi\bigl(1+\varepsilon m(a)\bigr)
=p^{\kappa}\sum_{n\in R_{q-\kappa}}\chi\bigl(1+\varepsilon p^{\kappa}n\bigr)=0,
\]
the inner sum being a complete sum of a character of
\(\bigl(1+p^{r+\kappa}\Zp\bigr)/\bigl(1+p^{q+r}\Zp\bigr)\) which is nontrivial
there because \(\chi\) is faithful on
\(\bigl(1+p^{r}\Zp\bigr)/\bigl(1+p^{q+r}\Zp\bigr)\) and \(\kappa<q\).
\end{proof}

\section{The automorphism group}\label{sec:auto}

This section contains the main theorem. We write the rows of \(H\) as functions
\[
H_a(b)=\chi(1+\varepsilon ab),\qquad a\in\Rq ,
\]
and note \(H_0\equiv1\) and \(H_a(0)=1\).

\begin{definition}\label{def:autgroup}
A \emph{projective monomial automorphism} of \(H\) is a quadruple
\((\pi,\sigma,\alpha,\omega)\) consisting of permutations \(\pi,\sigma\) of
\(\Rq\) and unimodular functions \(\alpha,\omega\) on \(\Rq\) such that
\[
H_a\bigl(\sigma(b)\bigr)\,\omega(b)=\alpha_a\,H_{\pi(a)}(b)
\qquad\text{for all }a,b\in\Rq .
\tag{3.1}\label{eq:autdef}
\]
Two such quadruples are identified when they differ by a global scalar. The set
of these, with composition, is the group \(\Mgp_{p;r,q}\).
\end{definition}

\begin{theorem}[Complete classification]\label{thm:auto}
Let \(p\) be odd and \(r,q\ge1\). Every projective monomial automorphism of
\(H_{p;r,q}\) is of the form
\[
\sigma(b)=\frac{s+\lambda b}{1+\varepsilon cb},
\qquad
\pi(a)=\frac{c+\lambda a}{1+\varepsilon sa},
\qquad
\omega(b)=H_c(b),\qquad \alpha_a=H_a(s),
\tag{3.2}\label{eq:normalform}
\]
for a unique triple \((c,s,\lambda)\in\Rq\times\Rq\times\Rq^{\times}\), and every
such triple defines an automorphism. Hence
\[
\Mgp_{p;r,q}\;\cong\;\bigl\{(c,s,\lambda):c,s\in\Rq,\ \lambda\in\Rq^{\times}\bigr\},
\qquad
\bigl|\Mgp_{p;r,q}\bigr|=p^{2q}\varphi(p^{q}).
\]
The composition law is
\[
(c,s,\lambda)\cdot(c',s',\lambda')
=\Bigl(\tfrac{c+\lambda c'}{u},\ \tfrac{\lambda's+s'}{u},\
\tfrac{\lambda\lambda'+\varepsilon cs'}{u}\Bigr),
\qquad u=1+\varepsilon sc' .
\tag{3.3}\label{eq:complaw}
\]
\end{theorem}

\begin{proof}
\emph{Normalization.} Put \(a=0\) in \eqref{eq:autdef}. Since \(H_0\equiv1\) we
get \(\omega(b)=\alpha_0H_{\pi(0)}(b)\); rescaling by the global scalar
\(\alpha_0^{-1}\) we may assume \(\alpha_0=1\), so that
\[
\omega(b)=H_c(b),\qquad c:=\pi(0).
\]
Put \(b=0\). Since \(H_{\pi(a)}(0)=1\) and \(\omega(0)=1\) we get
\[
\alpha_a=H_a(s),\qquad s:=\sigma(0).
\]

\emph{The exact identity.} Substituting these into \eqref{eq:autdef} and using
multiplicativity of \(\chi\),
\[
\chi\Bigl(\bigl(1+\varepsilon a\sigma(b)\bigr)\bigl(1+\varepsilon cb\bigr)\Bigr)
=\chi\Bigl(\bigl(1+\varepsilon as\bigr)\bigl(1+\varepsilon \pi(a)b\bigr)\Bigr).
\]
Both arguments lie in \(1+p^{r}\Zp\), and \(\chi\) is faithful on
\(\bigl(1+p^{r}\Zp\bigr)/\bigl(1+p^{q+r}\Zp\bigr)\); therefore
\[
\bigl(1+\varepsilon a\sigma(b)\bigr)\bigl(1+\varepsilon cb\bigr)
\equiv\bigl(1+\varepsilon as\bigr)\bigl(1+\varepsilon\pi(a)b\bigr)
\pmod{p^{q+r}} .
\tag{3.4}\label{eq:exact}
\]
This is the only consequence of \eqref{eq:autdef} that we use, and it is
equivalent to it.

\emph{Step 1: \(\sigma\) is determined.} Set \(a=1\) in \eqref{eq:exact} and
define \(\lambda:=\pi(1)\bigl(1+\varepsilon s\bigr)-c\). Expanding the
right-hand side,
\[
\bigl(1+\varepsilon s\bigr)\bigl(1+\varepsilon\pi(1)b\bigr)
=1+\varepsilon s+\varepsilon b\,\pi(1)\bigl(1+\varepsilon s\bigr)
=1+\varepsilon s+\varepsilon b\,(c+\lambda),
\]
while the left-hand side is
\(1+\varepsilon\sigma(b)\bigl(1+\varepsilon cb\bigr)+\varepsilon cb\). Comparing
and cancelling \(\varepsilon\),
\[
\sigma(b)\bigl(1+\varepsilon cb\bigr)=s+\lambda b ,
\]
which is the first half of \eqref{eq:normalform}, the denominator being a
principal unit and hence invertible.

\emph{Step 2: \(\pi\) is determined.} Substituting the formula for \(\sigma\)
back into the left-hand side of \eqref{eq:exact} for general \(a\),
\[
\bigl(1+\varepsilon a\sigma(b)\bigr)\bigl(1+\varepsilon cb\bigr)
=1+\varepsilon a\bigl(s+\lambda b\bigr)+\varepsilon cb
=1+\varepsilon as+\varepsilon b\,(a\lambda+c),
\]
whereas the right-hand side equals
\(1+\varepsilon as+\varepsilon\pi(a)b\bigl(1+\varepsilon as\bigr)\). Since this
holds for all \(b\), we get \(\pi(a)\bigl(1+\varepsilon sa\bigr)=c+\lambda a\),
which is the second half of \eqref{eq:normalform}.

\emph{Step 3: \(\lambda\) is a unit, and conversely.} The map \(\sigma\) is the
M\"obius transformation with matrix
\(N=\left(\begin{smallmatrix}\lambda&s\\ \varepsilon c&1\end{smallmatrix}\right)\),
of determinant \(\lambda-\varepsilon cs\). A M\"obius map over the local ring
\(\Rq\) is a bijection precisely when its determinant is a unit. Since
\(\varepsilon=p^{r}\) with \(r\ge1\) we have
\(\lambda-\varepsilon cs\equiv\lambda\pmod p\), so the determinant is a unit if
and only if \(\lambda\in\Rq^{\times}\). Conversely, given any
\((c,s,\lambda)\in\Rq\times\Rq\times\Rq^{\times}\), the formulas
\eqref{eq:normalform} define permutations \(\sigma,\pi\) and unimodular
\(\alpha,\omega\); reversing Steps 1--2 shows \eqref{eq:exact} holds identically,
hence so does \eqref{eq:autdef}. Uniqueness of the triple is immediate from
\(c=\pi(0)\), \(s=\sigma(0)\), and
\(\lambda=\pi(1)(1+\varepsilon s)-c\).

\emph{Composition.} Representing \(\pi\) by
\(M=\left(\begin{smallmatrix}\lambda&c\\ \varepsilon s&1\end{smallmatrix}\right)\)
and composing, \(MM'\) has lower-right entry \(u=1+\varepsilon sc'\), a unit;
renormalizing that entry to \(1\) gives \eqref{eq:complaw}.
\end{proof}

\begin{corollary}[Profinite limit]\label{cor:profinite}
The maps \eqref{eq:normalform} are compatible with reduction \(\Rq\to R_{q-1}\),
and the inverse limit over \(q\) is
\[
\varprojlim_q\Mgp_{p;r,q}
\;\cong\;
\left\{\begin{pmatrix}\lambda&s\\ p^{r}c&1\end{pmatrix}
: c,s\in\Zp,\ \lambda\in\Zp^{\times}\right\}\subseteq\PGL_2(\Zp),
\]
the congruence subgroup cut out by \(v_p(\text{lower-left})\ge r\), that is a
\(\Gamma_0(p^{r})\)-type subgroup.
\end{corollary}

\begin{proof}
Compatibility is immediate from \eqref{eq:normalform}. For the identification,
note that a matrix \(\left(\begin{smallmatrix}a&b\\ c&d\end{smallmatrix}\right)\)
over \(\Zp\) with \(v_p(c)\ge r\ge1\) and unit determinant has \(d\) a unit,
since \(\det\equiv ad\pmod{p^{r}}\); dividing through by \(d\) puts it in the
displayed normal form.
\end{proof}

\begin{remark}\label{rem:nodepth}
We do not identify this subgroup with a Moy--Prasad filtration term
\cite{MoyPrasad}, which would require a specified point of the Bruhat--Tits
building and a filtration parameter. The statement here is the explicit
congruence condition and nothing more.
\end{remark}

\section{Consequences of the classification}\label{sec:consequences}

\subsection{The character-table threshold}

\begin{definition}\label{def:closure}
Say that \(H\) has \emph{projective row-ratio closure} at base row \(a_0\) if for
all \(a_1,a_2\in\Rq\) there are \(a_3\in\Rq\) and a unimodular constant
\(\gamma=\gamma(a_0,a_1,a_2)\) such that, pointwise in the column coordinate,
\[
\bigl(H_{a_1}/H_{a_0}\bigr)\cdot\bigl(H_{a_2}/H_{a_0}\bigr)
=\gamma\,H_{a_3}/H_{a_0}.
\]
The terminology is introduced here; this projective form is the
monomial-equivalence invariant version of the row-group property of a character
table.
\end{definition}

\begin{lemma}\label{lem:closure}
Projective row-ratio closure at \(a_0\) holds if and only if
\[
p^{r}(a_1-a_0)(a_2-a_0)\equiv0 \pmod{p^{q}}
\qquad\text{for all }a_1,a_2\in\Rq .
\]
It is preserved by monomial equivalence.
\end{lemma}

\begin{proof}
For the displayed dephased matrix, evaluation at \(b=0\) forces
\(\gamma=1\).  Multiplicativity and faithfulness of \(\chi\), exactly as in the
proof of Theorem~\ref{thm:auto}, then reduce the closure requirement to a
congruence modulo \(p^{q+r}\) between products of principal units; expanding and
cancelling gives the displayed condition.  Under monomial equivalence, column
scalings cancel in every row ratio, row scalings contribute only a unimodular
constant \(\gamma\), and row/column permutations merely relabel the rows and the
pointwise coordinate.  Thus projective row-ratio closure is invariant.
\end{proof}

\begin{theorem}[Sharp character-table transition]\label{thm:chartable}
\(H_{p;r,q}\) is monomially equivalent to the character table of a finite abelian
group if and only if \(q\le r\).
\end{theorem}

\begin{proof}
If \(q\le r\) then \(p^{r}\equiv0\) in \(\Rq\), so \(\ell_{r,q}(ab)\equiv ab\) and
\(H(a,b)=\omega^{ab}\) is the character table of \(\Z/p^{q}\Z\).

Conversely, the character table of a finite abelian group has projective
row-ratio closure at every base row. If \(q>r\), choose \(a_1=a_0+1\) and
\(a_2=a_0+1\); then \((a_1-a_0)(a_2-a_0)=1\) is a unit, so
\(p^{r}\cdot1\not\equiv0\) modulo \(p^{q}\) and Lemma~\ref{lem:closure} shows
closure fails at \(a_0\). As \(a_0\) was arbitrary, closure fails at every base
row, and by Lemma~\ref{lem:closure} this obstruction is a monomial-equivalence
invariant.
\end{proof}

\subsection{The cocycle threshold}

\begin{theorem}[Canonical cocycle threshold]\label{thm:cocycle}
Let \(\psi(a,b)=\chi(1+\varepsilon ab)\), a normalized \(2\)-cochain on the
additive group \(\Rq\), and let
\[
\mathfrak A(a,b,c)=\frac{\psi(a,b)\,\psi(a+b,c)}{\psi(a,b+c)\,\psi(b,c)}
\]
be its associator. Then
\[
\mathfrak A(a,b,c)=1
\quad\Longleftrightarrow\quad
p^{r}abc(a-c)\equiv0\pmod{p^{q}} ,
\]
and consequently \(\psi\) is a \(2\)-cocycle on \(\Rq\) if and only if
\(q\le r\). Moreover no normalized \(1\)-cochain repairs the failure, and no
automorphism of the cyclic group \(\Rq\) removes it.
\end{theorem}

\begin{proof}
Put \(L=(1+\varepsilon ab)(1+\varepsilon(a+b)c)\) and
\(R=(1+\varepsilon a(b+c))(1+\varepsilon bc)\), so that
\(\mathfrak A=\chi(L/R)\) by multiplicativity. Direct expansion gives the exact
identity
\[
L-R=\varepsilon^{2}abc(a-c) .
\tag{4.1}\label{eq:assoc}
\]
By faithfulness, \(\mathfrak A=1\) if and only if \(L\equiv R\) modulo
\(p^{q+r}\), which after dividing \eqref{eq:assoc} by \(\varepsilon=p^{r}\) is
the displayed criterion.

If \(q\le r\) then \(p^{r}\equiv0\) in \(\Rq\) and the criterion holds
identically. If \(q>r\), take \((a,b,c)=(1,1,2)\); since \(p\) is odd,
\(abc(a-c)=-2\) is a unit, so \(p^{r}abc(a-c)\not\equiv0\) and \(\psi\) is not a
cocycle. The first failure is thus at \(q=r+1\).

For the last two assertions: for any normalized \(1\)-cochain \(f\) the
coboundary \(\delta f\) satisfies \(\delta(\delta f)=1\) identically, whence
\(\delta(\psi\cdot\delta f)=\delta\psi=\mathfrak A\); and under \(a\mapsto ua\)
with \(u\in\Rq^{\times}\) the obstruction transforms by
\(abc(a-c)\mapsto u^{4}abc(a-c)\), a unit multiple.
\end{proof}

\begin{remark}[Scope]\label{rem:cocyclescope}
Theorem~\ref{thm:cocycle} concerns the canonical cyclic indexing of \(H\) by
\(\Rq\). It does not exclude cocyclic realizations over other groups or other
indexings; the cocyclic framework
\cite{Flannery,PereraHoradam,EganFlanneryOCathain} supplies tests that we have
not carried out, and we regard the general cocyclicity of the family as open.
\end{remark}

\subsection{The Weyl-centrality threshold}

For \(c,s\in\Rq\) define operators on functions \(f:\Rq\to\mathbf C\) by
\[
(C_cf)(b)=\chi(1+\varepsilon cb)\,f\!\left(\frac{b}{1+\varepsilon cb}\right),
\qquad
(S_sf)(b)=f(b+s),
\]
so that \(C_cH_a=H_{a+c}\); these are the images in \(\Mgp_{p;r,q}\) of the
triples \((c,0,1)\) and \((0,s,1)\) of Theorem~\ref{thm:auto}.

\begin{theorem}[Canonical Weyl centrality]\label{thm:weyl}
The group commutator \(C_cS_sC_c^{-1}S_s^{-1}\) acts projectively by the
M\"obius matrix
\[
\begin{pmatrix}
1-\varepsilon cs & \varepsilon cs^{2}\\
-\varepsilon^{2}c^{2}s & 1+\varepsilon cs+\varepsilon^{2}c^{2}s^{2}
\end{pmatrix},
\]
and is projectively scalar if and only if \(\varepsilon cs\equiv0\pmod{p^{q}}\),
that is, if and only if \(v_p(cs)\ge q-r\). Consequently the commutator is
scalar for all \(c,s\) if and only if \(q\le r\), and for \(q>r\) it is scalar
exactly on the rectangles \(p^{\alpha}\Rq\times p^{\beta}\Rq\) with
\(r+\alpha+\beta\ge q\).
\end{theorem}

\begin{proof}
By Theorem~\ref{thm:auto} the column actions of \(C_c\) and \(S_s\) are the
M\"obius maps with matrices
\(L_c=\left(\begin{smallmatrix}1&0\\ \varepsilon c&1\end{smallmatrix}\right)\)
and
\(U_s=\left(\begin{smallmatrix}1&s\\ 0&1\end{smallmatrix}\right)\), and the
displayed matrix is \(L_cU_sL_c^{-1}U_s^{-1}\), computed directly. It is scalar
in \(\PGL_2(\Rq)\) if and only if both off-diagonal entries vanish and the two
diagonal entries agree; the diagonal difference is
\(\varepsilon cs(\varepsilon cs+2)\), and since \(p\) is odd and
\(\varepsilon cs\) is not a unit, \(\varepsilon cs+2\) is a unit. Hence
scalarity is equivalent to \(\varepsilon cs\equiv0\), and the off-diagonal
entries \(\varepsilon cs^{2}\) and \(-\varepsilon^{2}c^{2}s\) then vanish as
well. The remaining statements are the translation of
\(v_p(\varepsilon cs)\ge q\) into \(v_p(c)+v_p(s)\ge q-r\).
\end{proof}

\begin{remark}[Scope]\label{rem:weylscope}
Theorem~\ref{thm:weyl} concerns the canonical lower and upper torsors only. It
does not exclude noncanonical Heisenberg models built from other subgroups of
\(\Mgp_{p;r,q}\).
\end{remark}

\section{Polarized Heisenberg windows and their normalizers}\label{sec:polarized}

Theorem~\ref{thm:weyl} says that global scalar commutators disappear when
\(q>r\), but it also gives the exact valuation region on which scalarity
survives.  We now classify the maximal rectangles in that region and the
normalizers that preserve them.  These statements concern only the canonical
lower and upper torsors above.

For \(0\le\alpha,\beta\le q\), put
\[
A_\alpha=p^\alpha\Rq,\qquad B_\beta=p^\beta\Rq,
\]
with \(p^q\Rq=\{0\}\), and write
\[
\mathbf B(c,s)=\chi(1+p^rcs).
\]

\begin{theorem}[Maximal polarized Heisenberg windows]\label{thm:polarized}
The canonical commutators are scalar for every
\((c,s)\in A_\alpha\times B_\beta\) if and only if
\[
\boxed{\;r+\alpha+\beta\ge q.\;}
\]
On every such scalar window, \(\mathbf B\) is a bicharacter and
\[
\operatorname{Rad}_{A}(\mathbf B)
=p^{\max(\alpha,q-\beta)}\Rq,
\qquad
\operatorname{Rad}_{B}(\mathbf B)
=p^{\max(\beta,q-\alpha)}\Rq.
\]
If \(q\le r\), the full window \((\alpha,\beta)=(0,0)\) is scalar.  If
\(q>r\), the inclusion-maximal scalar windows are exactly
\[
\boxed{\;\alpha+\beta=q-r,\qquad \alpha,\beta\ge0,\;}
\]
so there are precisely \(q-r+1\) maximal polarizations.  On such a maximal
window the radicals are
\[
\operatorname{Rad}_{A}=p^{\alpha+r}\Rq,
\qquad
\operatorname{Rad}_{B}=p^{\beta+r}\Rq,
\]
and the residual groups
\[
\overline A_\alpha=A_\alpha/p^{\alpha+r}\Rq,
\qquad
\overline B_\beta=B_\beta/p^{\beta+r}\Rq
\]
are cyclic of order \(p^r\).  If \(\kappa\in\Rq^\times\) is the conductor unit
of the faithful character, then for
\(c=p^\alpha x\), \(s=p^\beta y\) one has
\[
\boxed{\;
\mathbf B(c,s)=\exp\!\left(\frac{2\pi i\,\kappa xy}{p^r}\right).
\;}
\]
Hence \(\mathbf B\) descends to a perfect pairing
\[
\overline A_\alpha\times\overline B_\beta\longrightarrow\mu_{p^r},
\]
which, up to the conductor unit, is the standard pairing on
\((\Z/p^r\Z)^2\).  Its commutator data therefore define the standard finite
Heisenberg central extension
\[
1\longrightarrow\mu_{p^r}\longrightarrow\mathsf H_{p^r}
\longrightarrow\overline A_\alpha\oplus\overline B_\beta
\longrightarrow1
\]
of order \(p^{3r}\).
\end{theorem}

\begin{proof}
The scalar-window criterion is Theorem~\ref{thm:weyl}: the least possible
valuation of a nonzero product in \(A_\alpha B_\beta\) is
\(\alpha+\beta\), so scalarity for the entire rectangle is equivalent to
\(r+\alpha+\beta\ge q\).  Inclusion of the parameter subgroups reverses both
exponents, hence for \(q>r\) a scalar rectangle is inclusion-maximal exactly on
the boundary \(\alpha+\beta=q-r\).

On a scalar window, multiplicativity in the first variable fails before
application of \(\chi\) only by the cross term
\(p^{2r}c_1c_2s^2\), whose valuation is at least
\(2r+2\alpha+2\beta\ge q+r\); the same argument applies in the second
variable.  Thus \(\mathbf B\) is a bicharacter.  Faithfulness gives
\(\mathbf B(c,s)=1\) exactly when \(cs\equiv0\pmod{p^q}\), which yields the
displayed radicals.  On a maximal window they reduce to
\(p^{\alpha+r}\Rq\) and \(p^{\beta+r}\Rq\), so both residual groups have
order \(p^r\).

Finally, on the boundary \(\alpha+\beta=q-r\) we have
\(cs=p^{q-r}xy\).  Under the normalized logarithm the higher terms are
invisible modulo \(p^q\), and
\[
p^{-r}\log(1+p^rcs)\equiv p^{q-r}xy\pmod{p^q}.
\]
Writing the faithful character through its conductor unit \(\kappa\) gives the
stated \(p^r\)-th root of unity.  The induced pairing has zero left and right
annihilator, so it is perfect and gives the standard finite Heisenberg central
extension.
\end{proof}

\begin{remark}[Scope of the residual Heisenberg statement]\label{rem:polarizedscope}
The theorem classifies valuation rectangles for the \emph{canonical} lower and
upper torsors.  A scalar window is not a new full automorphism group, and the
Heisenberg group above is the abstract central extension determined by the
perfect residual commutator pairing.  Central radical operators need not be
scalar before quotienting.  No arbitrary cocyclic model, noncanonical
Heisenberg subgroup, Weil representation, or metaplectic representation is
asserted.
\end{remark}

For a maximal polarization, write its projective window inside the normal form
of Theorem~\ref{thm:auto} as
\[
\mathsf W_{\alpha,\beta}
=\left\{\left[
\begin{pmatrix}1&s\\ p^rc&1\end{pmatrix}
\right]:c\in p^\alpha\Rq,\ s\in p^\beta\Rq\right\}.
\]
Because \(r+\alpha+\beta=q\), the lower and upper factors commute projectively,
and \(|\mathsf W_{\alpha,\beta}|=p^{q+r}\).

\begin{theorem}[Exact normalizers and polarization rigidity]\label{thm:normalizer}
Assume \(q>r\) and \(\alpha+\beta=q-r\).  For
\[
M(v,u,\lambda)=
\begin{pmatrix}\lambda&u\\p^rv&1\end{pmatrix},
\qquad \lambda\in\Rq^\times,
\]
the normalizer of \(\mathsf W_{\alpha,\beta}\) inside
\(\Mgp_{p;r,q}\) is exactly
\[
\boxed{\;
\mathsf N_{\alpha,\beta}
=\{M(v,u,\lambda):v\in p^\alpha\Rq,\ u\in p^\beta\Rq,
\ \lambda\in\Rq^\times\}.
\;}
\]
It has order
\[
|\mathsf N_{\alpha,\beta}|=p^{q+r}\varphi(p^q).
\]
If \(X(c,s)=\left(\begin{smallmatrix}1&s\\p^rc&1\end{smallmatrix}\right)\)
and \(\Delta=\lambda-p^ruv\), then exactly
\[
MX(c,s)M^{-1}=I+\Delta^{-1}
\begin{pmatrix}A&B\\C&-A\end{pmatrix},
\]
where
\[
A=p^r(uc-\lambda vs),\qquad
B=\lambda^2s-p^ru^2c,\qquad
C=p^r(c-p^rv^2s).
\]
For \(M\in\mathsf N_{\alpha,\beta}\) this reduces to
\[
\boxed{\;MX(c,s)M^{-1}=X(\lambda^{-1}c,\lambda s).\;}
\]
Consequently distinct maximal polarizations are pairwise nonconjugate inside
\(\Mgp_{p;r,q}\), and
\[
\boxed{\;\mathsf N_{\alpha,\beta}/\mathsf W_{\alpha,\beta}
\cong\Rq^\times.\;}
\]
On the perfect residual pairing this quotient acts by
\[
(x,y)\longmapsto(\lambda^{-1}x,\lambda y),
\]
so the effective action factors through \((\Z/p^r\Z)^\times\), with kernel
\(1+p^r\Rq\).
\end{theorem}

\begin{proof}
The displayed conjugation identity is direct matrix multiplication.  If a
conjugate is to lie in a window of the same form, its two diagonal entries must
agree projectively.  Since \(p\) is odd, this forces \(A=0\) for every
\(c\in p^\alpha\Rq\), \(s\in p^\beta\Rq\).  Taking first \(s=0\) and then
\(c=0\) yields
\[
p^ru p^\alpha=0,\qquad p^rv p^\beta=0\pmod{p^q},
\]
which, using \(\alpha+\beta=q-r\), is exactly
\(u\in p^\beta\Rq\), \(v\in p^\alpha\Rq\).  Conversely, under these
conditions \(\Delta\equiv\lambda\pmod{p^q}\), the nonlinear terms in \(B,C\)
vanish, and conjugation is precisely
\((c,s)\mapsto(\lambda^{-1}c,\lambda s)\).  This proves the normalizer formula
and the order count.

The same formula preserves the valuation ideals of \(c\) and \(s\), so a
maximal window cannot be conjugated to one with different \((\alpha,\beta)\).
Finally the map \(\mathsf N_{\alpha,\beta}\to\Rq^\times\) recording \(\lambda\)
is surjective with kernel \(\mathsf W_{\alpha,\beta}\), and its residual action
is the displayed inverse/direct scaling.  It depends only on \(\lambda\bmod
p^r\), giving the effective unit group \((\Z/p^r\Z)^\times\).
\end{proof}

\begin{corollary}[No internal building adjacency]\label{cor:nobuilding}
The \(q-r+1\) maximal polarizations are isolated conjugacy classes under the
canonical depth-\(r\) projective monomial group.  No element of that group moves
\((\alpha,\beta)\) to \((\alpha+1,\beta-1)\).  Thus these data alone do not
define a Bruhat--Tits edge or building adjacency.
\end{corollary}

\begin{remark}[Normalizer scope]\label{rem:normalizerscope}
The normalizer is classified only inside \(\Mgp_{p;r,q}\), equivalently inside
the finite depth-\(r\) projective monomial group.  Nonunit-determinant
correspondences and normalizers in a larger group such as \(\PGL_2(\mathbf Q_p)\)
are not covered.  Pairwise nonconjugacy is internal to this frozen group, and the
residual unit torus is one-dimensional; no full symplectic, Weil, metaplectic,
tree, or building theorem follows.
\end{remark}

\section{Separation from the Fourier matrix}\label{sec:abelian}

For \(q>r\) the matrices \(H_{p;r,q}\) and \(F_{p^{q}}\) have projective monomial
automorphism groups of the \emph{same order} \(p^{2q}\varphi(p^{q})\), so no
order-based invariant separates them. We separate them by abelianization.

\begin{proposition}[Group invariants are equivalence invariants]\label{prop:inv}
If \(K=AHB\) with \(A,B\) monomial, then conjugation by \((A,B)\) carries the
projective monomial automorphism group of \(H\) isomorphically onto that of
\(K\). Consequently every isomorphism invariant of that group --- its centre,
derived subgroup, exponent, nilpotency class, or abelianization --- is an
invariant of the monomial equivalence class of the matrix.
\end{proposition}

\begin{proof}
If \((\pi,\sigma,\alpha,\omega)\) fixes \(H\) projectively then the conjugated
quadruple fixes \(AHB\) projectively, and the assignment is a group isomorphism
with inverse given by conjugating back.
\end{proof}

We record Proposition~\ref{prop:inv} explicitly because the automorphism group is
often used only as an order-valued catalogue statistic; the separation below uses
a strictly finer invariant. For comparison, the standard separating invariants in
this literature are the defect \cite{TadejZyczkowskiDefect} and fingerprint-type
data \cite{Szollosi}.

\begin{theorem}[Abelianization separation]\label{thm:ab}
Write \(\varphi:\Mgp_{p;r,q}\to(\Z/p^{r}\Z)^{\times}\) for
\(\varphi(c,s,\lambda)=\lambda\bmod p^{r}\). Then \(\varphi\) is a surjective
homomorphism with \(\ker\varphi=[\Mgp,\Mgp]\), so
\[
\bigl|\Mgp_{p;r,q}^{\mathrm{ab}}\bigr|=\varphi(p^{r})=p^{r-1}(p-1),
\]
independently of \(q\). By contrast the projective monomial automorphism group
\(\mathfrak F_{p,q}\) of \(F_{p^{q}}\) has
\(\bigl|\mathfrak F_{p,q}^{\mathrm{ab}}\bigr|=p^{q-1}(p-1)\). Hence for \(q>r\)
the two groups are non-isomorphic and, by Proposition~\ref{prop:inv},
\(H_{p;r,q}\) and \(F_{p^{q}}\) are monomially inequivalent.
\end{theorem}

\begin{proof}
That \(\varphi\) is a homomorphism follows from \eqref{eq:complaw}: modulo
\(p^{r}\) the factor \(u=1+\varepsilon sc'\) is \(1\) and the term
\(\varepsilon cs'\) vanishes, so \(\lambda''\equiv\lambda\lambda'\). Hence
\([\Mgp,\Mgp]\subseteq\ker\varphi\).

For the reverse inclusion, write \(T_c=(c,0,1)\), \(S_s=(0,s,1)\) and
\(D_\lambda=(0,0,\lambda)\). From \eqref{eq:complaw} one computes
\[
[D_2,T_c]=T_c,\qquad [D_2,S_{-2x}]=S_x ,
\]
so every \(T\) and every \(S\) is a commutator; and \([T_1,S_1]\) is, after
factoring out its \(T\) and \(S\) parts, a dilation \(D_\mu\) with
\(v_p(\mu-1)=r\) exactly, so the dilations in \(\ker\varphi\) are generated by
commutators as well. Since every element of \(\ker\varphi\) factors as
\(T_cS_{x}D_{\lambda_1}\) with \(\lambda_1\equiv1\pmod{p^{r}}\), we conclude
\(\ker\varphi\subseteq[\Mgp,\Mgp]\).

For completeness, a Fourier automorphism may be written in the normal form
\(\pi(a)=x+\lambda a\), \(\sigma(b)=s+\lambda b\), with the corresponding row
and column phases determined by \(x,s\).  Reparameterize the column translation
by \(y=s/\lambda\).  In the coordinates \((x,y,\lambda)\) the group law is
\[
\mathfrak F_{p,q}\cong (\Rq\times\Rq)\rtimes\Rq^\times,
\qquad
(x,y,\lambda)(x',y',\lambda')
=(x+\lambda x',\ y+\lambda^{-1}y',\ \lambda\lambda').
\]
Thus the two translation coordinates carry the weights \(\lambda\) and
\(\lambda^{-1}\), and projection onto \(\Rq^\times\) is a surjective
homomorphism.  If \(D_2=(0,0,2)\), then
\[
[D_2,(x,0,1)]=(x,0,1),\qquad
[D_2,(0,-2y,1)]=(0,y,1),
\]
so both translation factors lie in the commutator subgroup (recall that \(p\)
is odd).  Conversely the quotient by the translations is the abelian group
\(\Rq^\times\).  Hence
\[
\mathfrak F_{p,q}^{\mathrm{ab}}\cong\Rq^\times,
\qquad
|\mathfrak F_{p,q}^{\mathrm{ab}}|=p^{q-1}(p-1).
\]
\end{proof}

\begin{remark}
The mechanism is worth isolating: the commutator subgroup of \(\Mgp\) reaches
exactly the dilations of level \(r\), so the abelianization \emph{saturates} at
\(\varphi(p^{r})\) while the Fourier abelianization grows with \(q\). It is the
depth \(r\) of the principal-unit character, not the size of the matrix, that is
visible in the abelianization.
\end{remark}

\subsection{Defect and a low-order fingerprint slice}\label{ssec:defect}

We complement the group-theoretic separation by computing two standard
catalogue invariants for a nonlinear member beyond order~\(9\).  For a dephased
complex Hadamard matrix \(K=(K_{ij})\) of order \(n\), write
\(K_{ij}e^{itR_{ij}}\) with \(R_{0j}=R_{i0}=0\).  Linearizing row
orthogonality gives, for every \(i<j\),
\[
 \sum_{k=0}^{n-1}K_{ik}\overline{K_{jk}}(R_{ik}-R_{jk})=0.
 \tag{6.1}\label{eq:defectsystem}
\]
The defect is the real nullity of this system in the \((n-1)^2\) dephased
variables \cite{TadejZyczkowskiDefect}.

\begin{proposition}[Reproducible computed invariants]\label{prop:defect313}
For the nonlinear order-\(27\) member \(H_{3;1,3}\),
\[
 \operatorname{def}(H_{3;1,3})=46,
 \qquad
 \operatorname{def}(F_{27})=28.
\]
Thus the defect independently separates \(H_{3;1,3}\) from the Fourier matrix.

The complete \(2\times2\) fingerprint slices of the two matrices agree.  More
precisely, every \(2\times2\) minor has magnitude
\[
 2\sin\!\left(\frac{\pi\delta}{27}\right),
 \qquad 0\le\delta\le13,
\]
and the common multiplicity \(m_\delta\) is
\[
 m_\delta=
 \begin{cases}
 5103,&\delta=0,\\
 19683,&\delta=9,\\
 13122,&\delta\in\{3,6,12\},\\
 6561,&\delta\in\{1,2,4,5,7,8,10,11,13\}.
 \end{cases}
\]
\end{proposition}

\begin{corollary}[Comparison of separating invariants]\label{cor:invcompare}
Among the four nonlinear members computed here, the standard invariants do not
suffice:
\[
\begin{array}{lccc}
 & \operatorname{def}(H_{p;r,q}) & \operatorname{def}(F_{p^{q}})
 & \text{separated?}\\\hline
H_{3;1,2}\ (n=9)   & 10 & 4  & \text{yes}\\
H_{3;1,3}\ (n=27)  & 46 & 28 & \text{yes}\\
H_{5;1,2}\ (n=25)  & 16 & 16 & \text{no}\\
H_{3;2,3}\ (n=27)  & 28 & 28 & \text{no}
\end{array}
\]
The defect therefore separates two of the four, and the \(2\times2\) fingerprint
separates none --- it agrees even in the separating case
\(H_{3;1,3}\) versus \(F_{27}\). By contrast Theorem~\ref{thm:ab} separates
\emph{every} nonlinear member from the corresponding Fourier matrix, since
\(\bigl|\Mgp^{\mathrm{ab}}_{p;r,q}\bigr|=\varphi(p^{r})\) is independent of
\(q\) while \(\bigl|\mathfrak F^{\mathrm{ab}}_{p,q}\bigr|=p^{q-1}(p-1)\) grows;
for the two undetected cases above the abelianization orders are \(4\) versus
\(20\), and \(6\) versus \(18\).
\end{corollary}

\begin{remark}\label{rem:whyab}
This is the reason Proposition~\ref{prop:inv} is stated explicitly. The
abelianization is not merely an alternative to the defect and the fingerprint on
this family: it is the only one of the three that works at every parameter, and
it does so by a theorem rather than a computation. The failures above are not
accidents of small size --- the defect and the fingerprint both measure local
deformation data, which the principal-unit family shares with the Fourier matrix
whenever the nonlinearity is confined to a single conductor level, whereas the
abelianization sees the depth \(r\) of the character directly.
\end{remark}

\begin{proof}[Computational certificate]
Modulo \(27\), the normalized logarithm is
\[
 \ell_{1,3}(z)=z+12z^2+3z^3.
\]
The script in Appendix~\ref{app:code} constructs the exponent matrix, forms the
real system \eqref{eq:defectsystem}, and computes its rank independently by
singular-value decomposition and pivoted QR.  For \(H_{3;1,3}\) the system has
\(676\) variables and rank \(630\); the smallest nonzero singular value is
\(5.5976\times10^{-1}\), while the largest singular value in the numerical
nullspace is about \(5.2\times10^{-15}\).  Pivoted QR gives the same rank at absolute
thresholds \(10^{-8},10^{-10},10^{-12}\).  The corresponding Fourier rank is
\(648\).

For rows \(a<b\) and columns \(c<d\), if
\[
 \Delta=e_{ac}+e_{bd}-e_{ad}-e_{bc}\pmod{27},
 \qquad
 \delta=\min(\Delta,27-\Delta),
\]
then the minor magnitude is exactly \(2\sin(\pi\delta/27)\).  Exhaustive
integer enumeration of all \(\binom{27}{2}^2=123201\) minors gives the displayed
multiplicities.  The calculation first recovers defect \(10\) for
\(H_{3;1,2}\), providing an internal validation against the known order-\(9\)
value.  The rank statement is a reproducible floating-point computation with a
large displayed gap, not a formal proof-assistant certification of algebraic
rank.
\end{proof}

\begin{remark}
The agreement of the \(2\times2\) fingerprint slices shows why a single
low-order fingerprint statistic is insufficient here.  Conversely, the defect
already separates \(H_{3;1,3}\) from \(F_{27}\), in agreement with the stronger
all-parameter abelianization separation of Theorem~\ref{thm:ab}.
\end{remark}

\section{Verification}\label{sec:verification}

Every structural theorem above has a proof valid for all odd \(p\) and all
\(r,q\ge1\); no theorem depends on the computations in this section.  We
separate the reproducible checks shipped with the publication archive from
additional development-time regressions that are recorded here only for
provenance.

\paragraph{Reproducible checks supplied with the archive.}
Appendix~\ref{app:code} contains an independent implementation of the defect and
fingerprint computations.  It reproduces the defects \(10,46,16,28\) at the
stated examples and the complete \(2\times2\) fingerprint slice for
\(H_{3;1,3}\) and \(F_{27}\).  Two additional exact-arithmetic scripts check the
polarized-window results.  The first verifies the scalar-window criterion,
maximal-polarization count, radical formulas, and perfect residual pairing on a
broad formula-level parameter box and exhaustively on the small rings listed in
the script.  The second checks the conjugation formula, normalizer criterion,
residual action, and order formulas by exact modular arithmetic on its explicitly
stated representative grids.  The latter is a sampled regression, not an
exhaustive finite-parameter proof; the all-parameter normalizer theorem is proved
in Theorem~\ref{thm:normalizer}.

\paragraph{Additional development-time checks not shipped as publication
artifacts.}
During development we also performed symbolic checks of orthogonality, the
row-ratio criterion, the normal form and composition law, and the associator and
Weyl-commutator identities; enumerated the automorphism group at
\((p,q)=(3,1),(3,2),(3,3)\), obtaining orders \(18,486,13122\); exhaustively
checked the cocycle threshold over the previously recorded small parameter
boxes; checked the centrality criterion on \(20{,}190\) parameter pairs; and
computed commutator-subgroup orders at \(q=1,2,3\).  Those development checks
are not part of the supplied reproducibility certificate and are not used as
proofs of any theorem.

\section{Open questions}\label{sec:open}

\begin{enumerate}
\item \emph{Equivalence orbits.} Only the order-\(9\) member \((3,1,2)\) has been
placed in a known orbit \cite{Karlsson}.  Proposition~\ref{prop:defect313}
gives defect \(46\) for \(H_{3;1,3}\), separating it from \(F_{27}\), but
does not identify its equivalence orbit.  The supplementary computations give
defects \(16\) and \(28\) for \((5,1,2)\) and \((3,2,3)\), equal to the
corresponding Fourier defects; higher fingerprint slices or direct equivalence
searches remain natural next steps.
\item \emph{General cocyclicity.} Theorem~\ref{thm:cocycle} rules out only the
canonical cyclic indexing. Is \(H_{p;r,q}\) cocyclic over some group for
\(q>r\)?
\item \emph{Noncanonical or global Heisenberg models.}
Theorem~\ref{thm:polarized} gives exact canonical residual Heisenberg quotients
on the maximal valuation windows when \(q>r\), while
Theorem~\ref{thm:normalizer} shows that those windows are pairwise nonconjugate
inside \(\Mgp_{p;r,q}\).  Are there natural noncanonical subgroups with scalar
commutator outside this valuation-polarized construction?
\item \emph{External polarization correspondences.}  Is there a natural
nonunit-determinant correspondence arising from the Hadamard kernel itself that
moves one maximal polarization to another?  Corollary~\ref{cor:nobuilding}
shows that no such adjacency exists inside the present projective monomial
group; without an operator acting on the kernel data, no building interpretation
is asserted.
\item \emph{Other primes and mixed moduli.} The construction uses only that
\(1+p^{r}\Zp\) is procyclic; the case \(p=2\) requires separate treatment.
\end{enumerate}

\appendix
\section{Reproducible invariant computation}\label{app:code}

The source archive accompanying this paper contains an independent implementation
of the invariant computation, a pinned requirements file, a recorded output
transcript, and SHA-256 checksums.  The implementation was reconstructed from
the displayed formulas after the earlier working script was unavailable; it
reproduces every numerical claim below, but no identity with that historical
script is asserted.  A clean run is:
\begin{lstlisting}[language=bash]
python -m venv .venv
. .venv/bin/activate
python -m pip install -r requirements.txt
python paperh_invariants.py | tee OUTPUT.txt
\end{lstlisting}
The recorded environment used Python~3.13.5, NumPy~2.3.5, and SciPy~1.17.0.
The program constructs the normalized logarithm directly from its finite
\(p\)-adic series, forms the dephased linearization \eqref{eq:defectsystem},
checks rank agreement by SVD and pivoted QR at three thresholds, and enumerates
all \(2\times2\) minors by exact exponent differences.

\begin{lstlisting}[language=Python,
caption={Independent implementation of the defect and fingerprint computations.},
label={lst:invariants}]
#!/usr/bin/env python3
"""Reconstructed publication verifier for Paper H catalogue invariants.

This is a fresh implementation from the formulas stated in the manuscript,
not the unrecovered historical script. It reproduces the defect and complete
2x2 fingerprint claims used in the paper.
"""
from __future__ import annotations

from collections import Counter
import math
import numpy as np
from scipy.linalg import qr


def v_p(n: int, p: int) -> int:
    if n == 0:
        return 10**9
    v = 0
    while n % p == 0:
        n //= p
        v += 1
    return v


def normalized_log_mod(p: int, r: int, q: int, z: int) -> int:
    """Return p^{-r} log(1+p^r z) modulo p^q for odd p.

    The n-th term has coefficient (-1)^(n+1) p^{r(n-1)}/n.
    We cancel the p-part of n before modular inversion of the prime-to-p part.
    For the small parameters used by the verification, 200 terms is far beyond
    the point where every possible contribution has p-adic valuation >= q.
    """
    mod = p**q
    z %= mod
    total = 0
    for n in range(1, 201):
        t = v_p(n, p)
        val = r * (n - 1) - t
        if val >= q:
            continue
        den0 = n // (p**t)
        num0 = p ** (r * (n - 1) - t)
        coeff = (num0 % mod) * pow(den0, -1, mod) % mod
        if n % 2 == 0:
            coeff = (-coeff) % mod
        total = (total + coeff * pow(z, n, mod)) % mod
    return total


def exponent_matrix(p: int, r: int, q: int) -> np.ndarray:
    n = p**q
    E = np.empty((n, n), dtype=np.int64)
    vals = [normalized_log_mod(p, r, q, z) for z in range(n)]
    for a in range(n):
        for b in range(n):
            E[a, b] = vals[(a * b) % n]
    return E


def fourier_exponent_matrix(n: int) -> np.ndarray:
    a = np.arange(n, dtype=np.int64)[:, None]
    b = np.arange(n, dtype=np.int64)[None, :]
    return (a * b) % n


def defect_system(E: np.ndarray) -> np.ndarray:
    n = int(E.shape[0])
    omega = np.exp(2j * np.pi / n)
    K = omega ** E
    nv = (n - 1) ** 2

    def idx(i: int, k: int) -> int:
        return (i - 1) * (n - 1) + (k - 1)

    rows = []
    for i in range(n):
        for j in range(i + 1, n):
            phase = K[i, :] * np.conjugate(K[j, :])
            row = np.zeros(nv, dtype=np.complex128)
            for k in range(1, n):
                if i:
                    row[idx(i, k)] += phase[k]
                if j:
                    row[idx(j, k)] -= phase[k]
            rows.append(row.real)
            rows.append(row.imag)
    return np.asarray(rows, dtype=float)


def defect_certificate(E: np.ndarray):
    A = defect_system(E)
    nvars = A.shape[1]
    singular = np.linalg.svd(A, compute_uv=False)
    rank_svd = int(np.sum(singular > 1e-10))

    # Pivoted QR on A.T has the same row rank and is cheaper than QR on A.
    _, R, _ = qr(A.T, mode="economic", pivoting=True)
    diag = np.abs(np.diag(R))
    qr_ranks = {thr: int(np.sum(diag > thr)) for thr in (1e-8, 1e-10, 1e-12)}

    nz_min = float(singular[rank_svd - 1]) if rank_svd else 0.0
    null_max = float(singular[rank_svd]) if rank_svd < len(singular) else 0.0
    return {
        "nvars": nvars,
        "rank_svd": rank_svd,
        "defect": nvars - rank_svd,
        "smallest_nonzero_sv": nz_min,
        "largest_null_sv": null_max,
        "qr_ranks": qr_ranks,
    }


def fingerprint(E: np.ndarray) -> Counter:
    n = int(E.shape[0])
    out = Counter()
    for a in range(n):
        for b in range(a + 1, n):
            for c in range(n):
                for d in range(c + 1, n):
                    delta = int((E[a, c] + E[b, d] - E[a, d] - E[b, c]) % n)
                    delta = min(delta, n - delta)
                    out[delta] += 1
    return out


def check(cond: bool, msg: str):
    print(("PASS" if cond else "FAIL") + ": " + msg)
    if not cond:
        raise AssertionError(msg)


def main():
    cases = [
        ((3, 1, 2), 10, 4),
        ((3, 1, 3), 46, 28),
        ((5, 1, 2), 16, 16),
        ((3, 2, 3), 28, 28),
    ]
    certs = {}
    for (p, r, q), expected_h, expected_f in cases:
        E = exponent_matrix(p, r, q)
        Hc = defect_certificate(E)
        Fc = defect_certificate(fourier_exponent_matrix(p**q))
        certs[(p, r, q)] = Hc
        check(Hc["defect"] == expected_h,
              f"defect H_{{{p};{r},{q}}} = {expected_h}")
        check(Fc["defect"] == expected_f,
              f"defect F_{p**q} = {expected_f}")

    c313 = certs[(3, 1, 3)]
    check(c313["rank_svd"] == 630, "H_{3;1,3} rank = 630")
    check(all(v == 630 for v in c313["qr_ranks"].values()),
          "pivoted-QR rank = 630 at 1e-8, 1e-10, 1e-12")
    print("H313_SINGULAR_GAP",
          f"smallest_nonzero={c313['smallest_nonzero_sv']:.12e}",
          f"largest_null={c313['largest_null_sv']:.12e}")

    E313 = exponent_matrix(3, 1, 3)
    fp_h = fingerprint(E313)
    fp_f = fingerprint(fourier_exponent_matrix(27))
    expected = Counter({
        0: 5103,
        9: 19683,
        3: 13122, 6: 13122, 12: 13122,
        1: 6561, 2: 6561, 4: 6561, 5: 6561, 7: 6561,
        8: 6561, 10: 6561, 11: 6561, 13: 6561,
    })
    check(fp_h == expected, "H_{3;1,3} complete 2x2 fingerprint multiplicities")
    check(fp_f == expected, "F_27 complete 2x2 fingerprint multiplicities")
    check(sum(fp_h.values()) == math.comb(27, 2) ** 2,
          "fingerprint counts all C(27,2)^2 minors")

    # Discriminating normalized-log polynomial stated in the manuscript.
    check(all(normalized_log_mod(3, 1, 3, z) == (z + 12*z*z + 3*z*z*z) % 27
              for z in range(27)),
          "ell_{1,3}(z)=z+12z^2+3z^3 mod 27")

    print("PAPERH_RECONSTRUCTED_INVARIANTS_ALL_CHECKS_PASSED")


if __name__ == "__main__":
    main()
\end{lstlisting}

\paragraph{Exact polarized-window checks.}
The publication archive also supplies the following exact verification scripts:
\begin{lstlisting}
code/verify_p762_polarized_windows.py
code/verify_p763_normalizers.py
\end{lstlisting}
They use exact modular arithmetic to replay the scalar-window/radical
classification and to test the normalizer/polarization-rigidity formulas on the
parameter grids stated in the scripts.  The P762 script includes exhaustive
small-ring checks; the P763 script is a sampled exact-arithmetic regression.
These computations verify the displayed formulas but are not substitutes for
the all-parameter proofs.

\sloppy

\end{document}